\documentclass[arxiv]{elsarticle}
\usepackage{amsmath, amsthm, amssymb, amsfonts, bbm}
\usepackage[utf8]{inputenc}
\usepackage{hyperref}
\usepackage[nolist]{acronym}
\usepackage{mathtools}
\usepackage{pgfplots}
\usepackage{pgfplotstable}
\usepackage{booktabs}
\usepackage{cleveref}

\usetikzlibrary{3d}
\usetikzlibrary{calc,arrows,arrows.meta,positioning,shapes,snakes,
	intersections,automata,petri,backgrounds}

\pgfplotstableset{
	every head row/.style={before row=\toprule,after row=\midrule},
	clear infinite
}

\newcommand{\rinst}{M_{\mathrm{inst}}}
\newcommand{\rcum}{M}

\DeclarePairedDelimiter{\norm}{\lVert}{\rVert}

\theoremstyle{theorem}
\newtheorem{lemma}{Lemma}[section]
\newtheorem{corollary}[lemma]{Corollary}
\newtheorem{theorem}[lemma]{Theorem}

\theoremstyle{definition}
\newtheorem{definition}[lemma]{Definition}

\theoremstyle{remark}
\newtheorem{remark}[lemma]{Remark}

\makeatletter
\@ifclasswith{elsarticle}{arxiv}{
\def\ps@pprintTitle{%
	\let\@oddhead\@empty
	\let\@evenhead\@empty
	\def\@oddfoot{\centerline{\thepage}}%
	\let\@evenfoot\@oddfoot}
}{}
\makeatother

\begin{document}

\begin{acronym}
	\acro{CTMC}[CTMC]{Continuous Time Markov Chain}
	\acro{MRP}[MRP]{Markov Reward Process}
	\acro{SPN}[SPN]{Stochastic Petri Net}
	\acro{PEPA}[PEPA]{Performance Evaluation Process Algebra}
\end{acronym}

\begin{frontmatter}

\title{%
	Computing performability measures in Markov chains by means of
	matrix functions\tnoteref{t1}}

\tnotetext[t1]{This research has been partially supported by the ISTI-CNR project ``TAPAS: Tensor algorithm for performability analysis of large systems'', 
by the INdAM/GNCS project 2018 ``Tecniche innovative
per problemi di algebra lineare'', and 
by the
Region of Tuscany (Project ``MOSCARDO - ICT technologies for
structural monitoring of age-old constructions based on wireless
sensor networks and drones'', 2016--2018, FAR FAS).%
}

\author[di,isti]{G.~Masetti}
\ead{giulio.masetti@isti.cnr.it}

\author[dm,isti]{L.~Robol\corref{cor1}\fnref{fn1}}
\ead{leonardo.robol@unipi.it}

\cortext[cor1]{Corresponding author}
\fntext[fn1]{This author is a member of the INdAM Research
	group GNCS.}
\address[di]{Department of Computer Science, Largo B. Pontecorvo 3, 
	Pisa, 56127, Italy.}
\address[dm]{Department of Mathematics, Largo B. Pontecorvo 5, 
	Pisa, 56127, Italy.}
\address[isti]{Institute of Science and Technology ``A. Faedo'', 
	Via G. Moruzzi, 1, 56124, Pisa, Italy.}


\begin{abstract}
We discuss the efficient computation of performance, 
reliability, and availability measures for Markov chains; these
metrics --- and the ones obtained by combining them, are often 
called performability measures. 

We show that this computational problem
can be recasted as the evaluation of a bilinear form induced by
appropriate matrix functions, and thus solved by leveraging the
fast methods available for this task. 

We provide a comprehensive analysis of the theory required to translate
the problem from the language of Markov chains to the one of matrix
functions. The advantages of this new formulation are discussed, 
and it is shown that this setting allows to easily study the
sensitivities of the measures with respect to the model
parameters. 

Numerical experiments confirm the effectiveness of our approach;
the tests we have run show that we can 
outperform the solvers available in state of the art commercial 
packages on a representative set of large scale examples.

\end{abstract}

\end{frontmatter}
\section{Introduction}
\label{sec:intro}

Performance and dependability models are ubiquitous~\cite{TB17} in
design and assessment of physical, cyber or cyber-physical systems and
a vast ecosystem of high level formalisms has been developed to
enhance the expressive power of \acp{CTMC}.  Examples include dialects
of \acp{SPN} such as Stochastic Reward Nets~\cite{TB17}, Queuing
networks~\cite{BOP01}, dialects of \ac{PEPA}~\cite{H96}, and more.
High level formalisms are to \ac{CTMC} what high level programming
languages are to machine code; in this setting, performance and
dependability measures are usually defined following high level
formalisms primitives.  Once model and measures have been defined,
automatic procedures synthesize, transparently to the modeler, a
\ac{CTMC} and a \emph{reward structure} on it, producing a \ac{MRP},
and the measures of interest are derived as a function of the reward
structure.

The main contribution of this paper is to show that the computation of these
measures can be recasted in the framework of 
\emph{matrix functions}, and therefore enable the use
of fast Krylov-based methods for their computation. In
particular, we provide a translation table that directly maps
common performability measures to their matrix function formulation. Moreover,
our approach is easily extendable to other kinds of measures. 
Matrix functions are a fundamental tool in numerical analysis, and arise in 
different areas of applied mathematics \cite{higham2008functions}. 
For instance, they are used in the
evaluation of centrality measures 
for complex networks \cite{estrada2010network}, 
in the computation of 
geometric matrix means that find applications in radar \cite{barbaresco2009new}
and image processing techniques \cite{fletcher2007riemannian,rathi2007segmenting}, 
as well as in the study and efficient solution of system of ODEs \cite{al2011computing,hochbruck1998exponential} and PDEs \cite{yang2011novel}. 

Many numerical problems in these settings can
be rephrased as the evaluation
of a bilinear form $g(v, w) = v^T f(A) w$, where $f(\cdot)$ is an 
assigned (matrix) function, $A$ a matrix, and $v$ and $w$ vectors. 
Often, the interest is in the approximation of $g(v, w)$ for a specific
choice of the arguments $v, w$, and so an explicit computation of $f(A)$
is both unnecessary and too expensive. 

Therefore, one has to resort to more efficient methods, trying to exploit
the structure of $A$ when this is available. For instance, in \cite{fenu2013network}
the authors describe an application to network analysis that involves
a symmetric $A$ (the adjacency matrix of an undirected graph),
 and they propose
 a Gauss quadrature scheme that provides guaranteed lower and 
upper bounds for the value of $g(v, w)$. Other approaches exploiting
banded and rank structures in the matrices, often encountered
in Markov chains, can be found in
\cite{benzi2014decay,massei2017decay}. These properties have already
been exploited for the steady-state analysis of QBD processes \cite{bini2017efficient,bini2017decay}.

We prove that performability measures defined in \ac{CTMC}s can be
rephrased in this form as well. In particular they can be written as
$g(v, w) = v^T f(Q) w$, where $Q$ is the infinitesimal generator of
the Markov chain (also called Markov chain transition matrix), 
and $f(\cdot), Q, v$, and $w$ are chosen
appropriately.  \acp{CTMC} arise when modeling the behavior of
resource sharing systems~\cite{BOP01}, software, hardware or
cyber-physical systems~\cite{TB17}, portfolio optimization \cite{dai},
and the evaluation of performance~\cite{BOP01},
dependability~\cite{ALRL04,T01} and performability~\cite{M82} measures
that we are going to rewrite as bilinear forms $g(v, w)$. These models
are obtained with different high-level formalism; however, all of them
are eventually represented as CTMCs.

A simple example, which is analyzed in more detail in Section~\ref{sec:example1},
can be constructed by a set of $9$ states, numbered from $0$ to $8$, 
assuming that we can transition from state $i$ to $i+1$ with an exponential
distribution with 
parameter $\rho_1$, and from $i$ to $i-1$ with rate $\rho_2$. Pictorially,
this can be represented using the following reachability graph:
\begin{center}
\begin{tikzpicture}
\foreach \x in {0, ..., 8} {
	\draw (\x, 0) circle (.3);
	\node at (\x, 0) {$\x$};
}
\foreach \x in {1, ..., 8} {
	\draw[->] plot[smooth] coordinates{(\x-.8,.3) (\x-.6,.4) (\x-.4,.4) (\x-.2,.3)};
	\node at (\x-.5,.7) {$\rho_1$};
	\draw[->] plot[smooth] coordinates{(\x-.2,-.3) (\x-.4,-.4) (\x-.6,-.4) (\x-.8,-.3)};
	\node at (\x-.5,-.7) {$\rho_2$};
}
\end{tikzpicture}
\end{center}
In this case, the infinitesimal generator matrix $Q$ has the following structure:
\[
Q = \begin{bmatrix}
-\rho_{2} & \rho_2 \\
\rho_1    & -(\rho_1 + \rho_2) & \rho_2 \\
& \ddots & \ddots & \ddots \\
& & \rho_1 & -(\rho_1 + \rho_2) & \rho_2 \\
& & & \rho_1 & -\rho_1 \\
\end{bmatrix},
\]
where $Q_{ij}$ is nonzero if and only if there is an edge from state $i$ to $j$.

Rephrasing performability measures using the language of matrix functions opens also new direction of research regarding how structures that are visible at the level of these high level formalism are reflected into the \ac{CTMC} infinitesimal generator, and then can be exploited to enhance measures' evaluation.

The paper is structured as follows. 
We start with a brief discussion of modeling and matrix function notations in Section~\ref{sec:notation}.
The main contribution is presented in Section~\ref{sec:examples}, 
where we give a dictionary for the conversion
of standard performance, dependability and performability 
measures in the parlance of matrix functions.

Using this new formulation, we present a new solution method 
in Section~\ref{sec:efficient}, and we demonstrate 
in Section~\ref{sec:sensitivity} that this framework
can be used to describe in an elegant and powerful form the \emph{sensitivity}
of the measures, which is directly connected to the Frech\'et derivative
of $f(\cdot)$ at $Q$. We show that, when an efficient scheme for the
evaluation of $g(v, w)$ is available, the sensitivity of the measures
can be estimated at almost no additional cost.

Finally, we perform numerical tests on three relevant case studies in Section~\ref{sec:numerical}. The new method is shown to be efficient with respect to state-of-the-art solution techniques implemented in commercial level software tools. We also test our method to
compute the sensitivity of measures as described in 
Section~\ref{sec:sensitivity}. The numerical experiments
demonstrate that the computation requires less than
twice the time needed to simply evaluate the measure.

To enhance readability, mathematical details are presented in the appendix.
In particular, a brief discussion of the spectral properties of $Q$ 
that are most relevant for our study is offered in appendix~\ref{sec:spectral}; then, 
we give a summary of the currently
available methods to compute performability measures
in appendix~\ref{sec:STDmethods} and standard notions about matrix functions in appendix~\ref{sec:basics}. 
Details about the translation to matrix functions are presented in appendix~\ref{sec:rephrasing}, where we provide a proof of 
the two main lemmas, Lemma~\ref{lem:inst} and Lemma~\ref{lem:acc},
and we present additional remarks about their application. 

\section{Performability measures as matrix functions}

\subsection{Model and notation}
\label{sec:notation}

We recall
that Markov chains are stochastic processes
with the \emph{memory-less} property, that is the
probability of jumping from state $i$ to state $j$ after some time $t$ 
depends only on $i$ and $j$, and not on the previous history. 
We consider \acp{CTMC} where the state space is finite so, without loss
of generality, we assume it to be $[n] := \{1, \ldots, n\}$. 

In addition, we assume that the probability of jumping from
$i$ to $j$ is distributed
with a given exponential rate $\lambda_{ij}$, so that we may define
a matrix $Q$ with entries 
\[
  Q_{ij} = \begin{cases}
    \lambda_{ij} & \text{if } i \neq j \\
    - (\lambda_{1i} + \ldots +  \lambda_{n,i})  & \text{if } i = j \\
  \end{cases}, 
\]
where we set $\lambda_{ii} = 0$ for any $i = 1, \ldots, n$. With this 
definition, given a certain initial probability distribution 
$\pi_0^T$, 
where the entry with index $i$ corresponds to the probability
of being in the state $i$, the probability at time $t$ can be 
expressed as 
\[
  \pi(t)^T = \pi_0^T e^{tQ}. 
\]
The stationary (or steady-state)
distribution $\pi$, that is the limit of $\pi(t)$
for $t \to \infty$ is guaranteed\footnote{The uniqueness and existence of $\pi$ is discussed
	in appendix~\ref{subsec:model} in further detail.} to exist if the process
is \emph{irreducible}. If the process has \emph{absorbing states},
then the stationary distribution might not exists or not be unique; in
this case, it is typically of interest to study the transient behavior
of the process. 

From the modeling perspective, the steady-state distribution 
describes the long-term behavior of a system. However, when
assessing the performance and reliability of processes modeling real-world
phenomena, it is essential to characterize the \emph{transient} phase as
well, that is the behavior between the initial configuration 
and the steady-state. Moreover, the transient state is relevant
also when $X(t)$ is reducible, even though
the steady-state distribution is not well-defined in this case.

In practice, if the system has a large number of states, computing
$\pi(t)$ at some time $t$ is not the desired measure; 
it is far more interesting to 
obtain concise information by ``postprocessing'' $\pi(t)$ in
an appropriate way. Note that, in this framework, it is expected
that $\pi(t)$ depends on the initial choice of $\pi_0$, so that is 
an important parameter that needs to be known. 

Let us make an example to further clarify this concept. Consider a
Markov chain $X(t)$, with infinitesimal generator $Q$, modeling a 
publicly available service. Assume that 
we can partition the states in two sets $\mathbf u$ and $\mathbf d$. 
The first contains the states where the system is \emph{online} (``up''),
whereas
the second the ones where the system is offline (``down''). Up to permuting the states, we can partition
the matrix $Q$ according to this splitting:
\[
  Q =
  \begin{bmatrix}
    Q_{\mathbf u} &  Q_{\mathbf u \mathbf d} \\
    Q_{\mathbf d \mathbf u} &  Q_{\mathbf d} \\
  \end{bmatrix}. 
\]
In a certain interval
of time $[0, t]$, we would like to know how long the system is expected
to stay online. This can be measured by computing 
integral
\[
  U(t) = \int_0^t \mathbb P\{ X(\tau) \in \mathbf u \}\ d\tau.
\]
We note that $\mathbb P\{ X(\tau) \in \mathbf u \} = 
\langle \pi(\tau) , \mathbbm 1_{\mathbf u} \rangle$, 
where $\mathbbm 1_{\mathbf u}$ is the vector with ones in the states
of $\mathbf u$, and $0$ otherwise, and 
$\langle \cdot, \cdot \rangle$ denote the usual scalar product. 
This is what we call the
\emph{uptime} of the system in $[0, t]$. If the set $\mathbf d$
is the set of absorbing states, then this measure coincides
with the \emph{reliability} of the system.

This is an instance of a broader
class of measures that belong to the field of
\emph{performance}, \emph{availability}
and \emph{reliability} modeling. In the literature, 
the terms performance and availability refer to measures
depending on the steady-state, whereas reliability concerns the
transient phase of a reducible chain. 
Measures combining performance and availability (or reliability 
in the reducible case) are called \emph{performability measures}
\cite{meyer1980evaluating}. 

This paper is concerned
with the efficient computation of such measures, which will be obtained
by rephrasing the problem as the evaluation of a bilinear form
induced by a matrix function. 

The main tool used to devise fast algorithms for the evaluation
of these measures is recasting them as the computation of matrix
functions. Informally, given any square matrix $A$ and a
function $f(z)$ which can be expanded as a power series,
a matrix function $f(A)$ is defined as:
\[
  f(A) := \sum_{j \geq 0} c_j A^j, \qquad
  \text{where }
  f(z) = \sum_{j \geq 0} c_j z^j. 
\]
The most well-known example is the case $f(A) = e^A$, where
$c_j = \frac{1}{j!}$. Another function of interest in this paper
is $\varphi_1(z) := (e^{z} - 1) / z$. 
The definition holds in a more general setting, as
discussed in appendix~\ref{sec:basics}.

Let us fix some notation. We denote by $r$ a fixed weight
vector of $n$ elements, so that $r_{X(t)}$ is a process
whose value at time $t$ corresponds to entry of index $X(t)$ 
in $r$. In most cases, $r$ will be the \emph{reward vector}, containing a ``prize'' assigned for being a certain state. 
We shall give two definitions of reward measures, to simplify the
discussion of their computation later on. 

\begin{definition} \label{def:inst-perfomability}
	Let $r$ be a reward vector, and $X(t)$ a Markov chain with
	infinitesimal generator $Q$. Then, the number
	$\mathbb E[r_{X(t)}]$ is called \emph{instantaneous reward measure
		 at time $t$}, 
	and is denoted by $\rinst(t)$. Similarly, the number
	\[
	  \rcum(t) =  \int_{0}^t \mathbb E\left[ r_{X(\tau)} \right] \ d\tau = 
	  \int_{0}^t \rinst(\tau) \ d\tau
	\]
	is called \emph{cumulative reward measure at time $t$}. 
\end{definition}

Note that both definitions depend on the choice of the reward vector
$r$. This dependency is not explicit in our notation to make it more
readable --- in most of the following examples the current choice of $r$ 
will be clear from the context. Occasionally, we will make this 
dependency explicit by saying that a measure is associated
	with a reward vector $r$. 

Intuitively, the instantaneous reward measures
the probability of being in a certain set of states (the non-zero
entries of $r$), weighted according to the values of the components of $r$. 
The cumulative version is averaged over the time interval $[0, t]$.
For instance, for the uptime we had $r := \mathbbm 1_{\mathbf u}$. 

In the remaining part of this section, we show that reward measures
can be expressed as $\pi_0^T f(Q) r$, for a certain reward vector $r$ and 
an appropriate function $f(z)$. 

The next result is the first example of this construction, and 
concerns instantaneous reward measures. 

\begin{lemma} \label{lem:inst}
    Let $\rinst(t)$ be an instantaneous reward measure
    associated with a vector $r$ and a Markov process
    generated by $Q$ and with
    initial state $\pi_0$. 
    Then, $\rinst(t) = \pi_0^T f(Q) r$, with 
    $
      f(z) := e^{tz}. 
    $
\end{lemma}

\begin{proof}
	The equality follows immediately by the relation $\pi(t)^T = \pi_0^T e^{tQ}$. 
\end{proof}

A similar statement holds for the cumulative reward measure. 

\begin{lemma} \label{lem:acc}
    Let $\rcum$ be an cumulative reward measure, as 
    in Definition~\ref{def:inst-perfomability}. 
    Then, $\rcum(t) = \pi_0^T f(Q) r$, with 
    \[
    f(z) := t \varphi_1(tz) = \begin{cases}
     \frac{e^{tz} - 1}{z} & z \neq 0 \\
     t & z = 0 \\
    \end{cases}
    \]
  \end{lemma}

  \begin{proof}    
    The proof of this Lemma is given in \ref{sec:rephrasing}. 
  \end{proof}

\section{Translation of performability measures into matrix functions}
\label{sec:examples}

The purpose of this section is to construct a ready-to-use dictionary
for researchers involved in modeling that can be used to translate
several known measures to the matrix function formulation with little
effort. This is achieved by applying some theoretical results  which,
to ease the reading, are discussed in \ref{sec:rephrasing}. Following
these results, one can derive analogous formulations for additional
performability measures.

For many measures, it is important to identify a set of states
that correspond to the \emph{online} or \emph{up} state: when the system is 
in one of those states then it is functioning correctly. To keep
a uniform notation, we refer to this set as 
$\mathbf u \subseteq \{ 1, \ldots, n\}$. Its complement, 
the \emph{offline} or \emph{down} states, will be denoted 
by $\mathbf d$. In the case of reducible Markov chains, the
set of ``down'' states typically coincides with the absorbing
ones. This hypothesis is necessary for some measures (such as
the mean time to failure) in order to make them well-defined. 

Clearly,
the actual meaning of being ``up'' or ``down'' 
may change dramatically from
one setting to another; but the computations involved are
essentially unchanged --- and therefore we prefer to keep this
nomenclature to present a
unified treatment. 
A summary of all the different reformulations in this section, 
with references to the location where the details are discussed, 
is given in Table~\ref{tab:translations}. 

\subsection{Instantaneous reliability}
\label{sec:inst-reliability}

We consider the case of a reducible chain, with a set of absorbing
states (the ``down'' states). We are concerned with
determining the probability of being in an ``up'' state at any
time $t$. This measure is called \emph{instantaneous reliability}. 

This measure can be computed considering the probability of being
in a state included in the set $\mathbf u$, i.e., 
\[
  R(t) = \mathbb P\{ X(t) \in \mathbf u \} = \sum_{i \in \mathbf u} \pi_i(t)
    = \pi(t)^T r, 
\]
where $r = \mathbbm 1_{\mathbf u}$ 
is the vector with components $1$ on the indices in $\mathbf u$, 
and zero otherwise. Therefore, this availability measure is 
rephrased in matrix functions terms as $R(t) = \pi_0^T e^{tQ} r$. 

In the same way, we may define the measure $F(t) = 1 - R(t)$, which is
the probability of being in ``down'' state at the time $t$. Notice that 
this can also be expressed in matrix function form by 
$F(t) = \pi_0^T e^{tQ} (\mathbbm 1 - r)$. 

\subsection{Instantaneous availability}
\label{sec:inst-availability}

In irreducible Markov chains, the \emph{instantaneous availability}
is the analogue of the reliability described in the previous section, 
that is, we measure the probability of the system being ``up'' 
at any time $t$. 
We note that, mathematically, the definition of reliability and availability coincide 
but the former term is considered when dealing with reducible Markov chains, whereas the latter is employed for irreducible ones.
In particular, the availability can be expressed in matrix function form as follows:
$A(t) = \pi_0^T e^{tQ} \mathbbm 1_{\mathbf u}$. 

When the states in $\mathbf u$ correspond to the working 
state of at least $k$ components out of $n$, this measure is
often called the \emph{$k$-out-of-$n$ availability} of the system.

\begin{table}
	\centering
	\begin{tabular}{l|cccc}
		{Measure} & Function & Reward vector & Matrix & Reference \\ \toprule
		Inst. reliability & $e^{tz}$ & $\mathbbm 1_{\mathbf u}$ & $Q$ & Section~\ref{sec:inst-reliability} \rule{0pt}{1.3em}\\
		Inst. availability & $e^{tz}$ & $\mathbbm 1_{\mathbf u}$ &$Q$ 
		& Section~\ref{sec:inst-availability} \rule{0pt}{1.3em}\\
		MTTF & $-\frac{1}{z}, \ t\varphi_1(tz) $ & $\mathbbm 1_{\mathbf u}$ &$Q_{\mathbf u},\  Q$ & Section~\ref{sec:mttf}
		\rule{0pt}{1.3em} \\
		Exp. \# failures & $t\varphi_1(tz)$ & $\mathbbm 1$ & $Q_{\mathbf u \mathbf d}$ & Section~\ref{sec:expnumberfailures} 
		\rule{0pt}{1.3em}\\ 
		Uptime & $t\varphi_1(tz)$ & $\mathbbm 1_{\mathbf u}$ &$Q$ & Section~\ref{sec:uptime} 
		\rule{0pt}{1.3em}\\
		Average clients & $e^{tz}, \delta(z)$ & $[ 0, 1, \ldots, n-1 ]^T$ &$Q$ & Section~\ref{sec:clients}
		\rule{0pt}{1.3em}
	\end{tabular}
	
	\caption{Summary of the equivalence between performability measures
		and matrix functions, with the corresponding reward vector. The
		details on the interpretation of the set $\mathbf u$ and 
		on the reformulation are given in the linked 
		sections.}
	\label{tab:translations}
\end{table}

\subsection{Mean time to failure} 
\label{sec:mttf}

We consider the expected time of failure for a model. This measure
is relevant for devices which fail, and cannot be repaired. In particular, 
it is possible to exit from states in $\mathbf u$, but one can
never go back again: the Markov chain is reducible
and $\mathbf d$ is a set of absorbing states.
The average time needed to exit
$\mathbf u$ can be expressed as the average time that one spends
inside $\mathbf u$. In probabilistic terms, 
\[
  \mathrm{MTTF} = \int_0^{\infty} \mathbb E[(\mathbbm 1_{\mathbf u})_{X(\tau)}]\ d\tau,
\]
where $(\mathbbm 1_{\mathbf u})_{X(\tau)}$ denotes the component of index $X(\tau)$
in the vector $\mathbbm 1_{\mathbf u}$.
For this measure to
be finite, it is necessary that all the states inside 
of $\mathbf u$ have zero probability in the steady-state. 
This can be rephrased using $\pi(t)$ as follows:
\[
  \mathrm{MTTF} = \sum_{i \in \mathbf u} \int_{0}^{\infty} \pi_i(\tau)\ d\tau. 
\]
Here, one could be tempted to apply Lemma~\ref{lem:acc} directly, but 
this is not feasible. In fact, taking the limit of $t$ to $\infty$ 
for $f(z)$ gives $f(z) = z^{-1}$, which has a pole at $0$, 
and $Q$ is always singular. However, one can notice that, for $i \in \mathbf u$, 
we have $\pi_i(t) = ( \pi_{0,\mathbf u}^T e^{t Q_{\mathbf u}} )_i$ 
where $\pi_{0,\mathbf u}$ is the vector of initial conditions 
restricted to the indices in $\mathbf u$. Therefore, we can apply Lemma~\ref{lem:acc}
and take the limit of $t \to \infty$ to obtain: 
\[
  \mathrm{MTTF} = \pi_{0,\mathbf u}^T f(Q_{\mathbf u}) \mathbbm 1, \qquad 
  f(z) = -\frac{1}{z}, 
\]
which can be written simply as
$MTTF = - \pi_{0,\mathbf u}^T Q_{\mathbf u}^{-1} \mathbbm 1$. The
matrix $Q_{\mathbf u}$ is invertible\footnote{See Lemma~\ref{lem:m-matrix}}, its inverse is nonnegative
\cite{plemmons1981matrix}, and therefore $MTTF > 0$.

The same measure is often restricted to the interval
$[0, t]$. In this case, we may write
\[
  \mathrm{MTTF}(t) = \int_0^{t} \mathbb E[(\mathbbm 1_{\mathbf u})_{X(\tau)} ]\ d\tau. 
\]
We note that this formulation is well-defined even when $t$  goes to infinity. 
In fact, a direct application of Lemma~\ref{lem:acc} yields
\[
  \mathrm{MTTF}(t) = \pi_0^T f(Q) \mathbbm 1, \qquad 
  f(z) = t\varphi_1(tz) = \frac{e^{tz} - 1}{z}, 
\]
and this function does not have a pole in $0$. Nevertheless, 
also in this case it holds true that 
$\pi_0^T f(Q) \mathbbm 1 = \pi_{0,\mathbf u}^T f(Q_{\mathbf u}) \mathbbm 1$, 
and this gives a reduction in the size of the matrix whose exponential
needs to be computed, so this reformulation may be convenient in practice. 

\subsection{Expected number of failures}
\label{sec:expnumberfailures} 

We consider a system partitioned as usual in up and down states 
(denoted by $\mathbf u$ and $\mathbf d$). We are interested in 
computing the expected number of transitions between a state in 
$\mathbf u$ to a state in $\mathbf d$ (the number of failures). 

If we consider two states $i$ and $j$, then the expected 
number of transitions $N_{ij}(t)$ 
from $i$ to $j$ in a certain time
interval $[0, t]$ can be expressed as 
\[
  \mathbb{E}[N_{ij}(t)] = Q_{ij} \cdot \int_{0}^t \pi_i^T(\tau)\ d\tau. 
\]
When considering two sets of states, $\mathbf u$ and $\mathbf d$, 
this can be generalized to the expected number
of transitions from $\mathbf u$ to $\mathbf d$ by 
\[
  \int_0^t \pi(\tau) \begin{bmatrix}
  0 & Q_{\mathbf u \mathbf d} \\
  0 & 0 \\
  \end{bmatrix} \mathbbm 1\ d\tau = 
  \int_0^t \pi_{\mathbf u}(\tau) Q_{\mathbf u \mathbf d} \mathbbm 1\ d\tau
  = \pi_{0,\mathbf u}^T f(Q_{\mathbf u \mathbf d}) \mathbbm 1, 
\]
where $f(z) = \frac{e^{tz} - 1}{z}$, $\pi_{\mathbf u}(t)$ and $\pi_{0,\mathbf u}$ 
are the probability distributions restricted to the states in $\mathbf u$. 

\subsection{Uptime}
\label{sec:uptime}

The \emph{uptime} measure determines the expected 
availability of a system in a time interval $[0, t]$, 
for an irreducible Markov chain. To this end, 
we need to partition the states in online and offline, and to compute
the integral
\[
  U(t) = \int_{0}^t \mathbb E\left[ r_{X(\tau)} \right]\ d\tau, \qquad 
  r = \mathbbm{1}_{\mathbf u}. 
\]
We note that this is the integral analogous of the instantaneous 
availability defined for irreducible systems in Section~\ref{sec:inst-availability}. In fact, a straightforward
computation shows that 
\[
  U(t) = \int_0^t A(\tau)\ d\tau = \pi_0^T f(Q) \mathbbm 1_{\mathbf u}, 
  \qquad
  f(z) = t\varphi_1(tz) = \frac{e^{tz} - 1}{z}, 
\]
as predicted by Lemma~\ref{lem:acc}. We notice that this measure is 
not well-defined if we let $t$ go to infinity, since for 
every irreducible Markov chain 
$\lim_{t \to \infty} A(t) = \sum_{i \in \mathbf u} \pi_i > 0$, 
and therefore the limit of $U(t)$ needs to be infinite, because
the integrand is not infinitesimal. 

\subsection{Average number of clients} \label{sec:clients}

We now discuss
a measure which is specifically tailored to a model but, 
with the proper adjustments, can be made fit a broad number of settings. 
Assume we have a Markov chain $X(t)$ that models a queue (which might
be at some desk serving clients, a server running some software, 
or similar use cases). The state of $X(t)$ is the number of
clients waiting in the queue, and we assume a maximum number $n-1$ of 
slots. 
At any time, the process can finish
to serve a client with a rate $\rho_1$, or get a new client
in the queue with rate $\rho_2$. 

We are interested in the expected number of clients in the queue
at time $t > 0$, or at the steady state (that corresponds to $t \to \infty$). 
In the two cases, this measure can be expressed as 
\[
  \mathbb E[X(t)] = \pi(t)^T v, \qquad 
  \mathbb E[X(\infty)] = \pi^T v, \qquad 
  v = \begin{bmatrix}
    0 \\
    1 \\
    \vdots \\
    n - 1
  \end{bmatrix}, 
\]
where as usual we denote by $X(\infty)$ 
the limit of $X(t)$ to the steady-state. It
is clear that, since $\pi(t) = \pi_0^T e^{tQ}$, we can express
$\mathbb E[X(t)]$ as the bilinear form $\pi_0^T e^{tQ} v$. It is
interesting that one can express the steady state probability
in matrix function form as well. In fact, if we define $\delta (z)$
as the function equal to $1$ at $0$ and $0$ elsewhere, we can express\footnote{This is proven in
  Lemma~\ref{lem:steady-state}.}
$\mathbb E[X(\infty)]$ as $\pi_0^T \delta(Q) v$.

\section{Efficient computation of the measures}
\label{sec:efficient}

In view of the analysis of Section~\ref{sec:examples}, we are now aware that several 
measures associated with a Markov process
$X(t)$ are in fact computable by evaluating $w^T f(Q) v$, for appropriate
choices of $w, v$ and of the matrix function $f(Q)$. A straightforward
application of standard dense linear algebra
methods to compute $f(Q)$ usually has
complexity $\mathcal O(n^3)$, where $n$ is the size of the 
matrix $Q$, which in this case is the number of states of the underlying 
Markov chain. 

It is often recognized in the literature that the matrix $Q$ generating
the Markov chain is structured, and allows for a fast matrix 
vector product $v \mapsto Qv$. Typically, we can expect this operation
to cost $\mathcal O(n)$ flops, where $n$ is the number of states 
in the Markov chain. In this section we propose to leverage well
established Krylov approximation methods for the computation 
of $w^T f(Q)$, which in turn yields an algorithm
for evaluating $w^T f(Q) v$ in linear time and memory. Similar
ideas and techniques can be found in \emph{exponential integrators}, 
see \cite{gander2013paraexp}. 

Here we recall only the essential details needed to carry out 
the scheme, and we refer to \cite{frommer2014efficient} and
the references therein for further details. We now focus
on the computation of $f(Q) v$, ignoring $w$. Once this is
known, $w^T f(Q) v$ can be obtained in $\mathcal O(n)$ flops through
a scalar product. 

\subsection{Krylov subspace approximation}

The key ingredient to the fast approximation of $f(Q) v$ is the so-called 
Arnoldi process, the non-symmetric extension of the Lanczos scheme. 
From now on, we assume without loss of generality that $\norm{v}_2 = 1$. 
Consider the Krylov subspace of order $m$ generated by $Q$ and $v$
as 
\[
  \mathcal K_m(Q, v) := \mathrm{span}\{ v, Qv, Q^2v, \ldots, Q^{m-1} v\}. 
\]
Assuming no breakdown happens, the Arnoldi scheme provides an 
orthogonal basis $V_m$ for this space 
that satisfies the relation: 
\[
  QV_m = V_m H_m + h_{m+1,m} v_{m+1} e_m^T,
\]
where $H_m$ is an $m \times m$ upper Hessenberg matrix, $h_{m+1,m}$ a scalar,
$v_{m+1}$ a vector, and $e_m$ the $m$-th column of the identity matrix. This relation
is often used in the description of the classical Arnoldi method, and can be employed to
iteratively and efficiently construct the basis $V_m$. We refer the reader
to \cite{demmel1997} for further details.
This can be used to retrieve an approximation of $f(Q)v$ 
by computing $f_m = V_m f(H_m) V_m^T v = V_m f(H_m) e_1$. This approximation
has several neat features, among which we find the \emph{exactness} properties: 
the approximation $f_m$ is exact if $f(z)$ is a polynomial of degree at most 
$m - 1$. We would like to characterize how accurate is this approximation 
for generic function. To this aim, we introduce the following concept. 

\begin{definition}
  Given a matrix $Q$, the subset of the complex plane defined as 
  \[
    \mathcal W(Q) := \{ x^T Q x \ : \ \norm{x}_2 = 1 \}
  \]
  is called the \emph{field of values} of $Q$. 
\end{definition}

The above set is easily seen to be convex, and always contain the 
eigenvalues of $Q$. Whenever $Q$ is a normal matrix (for instance, 
when $Q$ is symmetric), then the set $\mathcal W(Q)$ is the convex
hull of the eigenvalues. For non-normal matrices, this set is
typically slightly larger, but the following relation holds: 
\[
  \sigma(Q) \subseteq \mathcal W(Q) \subseteq \{ 
    z \in \mathbb C \ : \ |z| \leq \norm{Q}_2 
  \} =: B(0, \norm{Q}_2), 
\]
where $\sigma(Q)$ is the spectrum, i.e., the set of eigenvalues, of $Q$.
so in particular the set is not unbounded and, in the Markov chain setting, 
it is typically to estimate $\norm{Q}_2$ to obtain a rough
approximation of its radius. 

\begin{lemma}
    Let $f(z)$ be a function defined on the field of values of $Q$, 
    and $p(z)$ a polynomial approximant of $f(z)$ such that 
    $|f(z) - p(z)| \leq \epsilon$ on $\mathcal W(Q)$. Then, 
    \[
      \norm{p(Q) - f(Q)}_2 \leq (1 + \sqrt{2}) \cdot \epsilon. 
    \]
\end{lemma}

\begin{proof}
    This inequality follows immediately from the well-known Crouzeix
    inequality (sometimes called Crouzeix conjecture, since the bound
    is conjectured to hold with $2$ in place of $1 + \sqrt{2}$). See, 
    for instance, \cite{crouzeix2017numerical}. 
\end{proof}

A straightforward implication of the above result is that if
$p_m(z)$ is a degree $m-1$ approximant to $f(z)$ and 
$|p_m - f| \leq \epsilon_m$
on a domain containing $\mathcal W(Q)$, then 
$\norm{f_m - f(Q) v}_2$ can be bounded by
writing $f(z) = p_m(z) + r_m(z)$ and 
using the exactness property:
\begin{align*}
  \norm{f_m - f(Q) v}_2 &\leq 
    \norm{V_m p_m(H_m) e_1 - V_m r_m(H_m) e_1 - 
    p_m(Q) v + r_m(Q) v}_2 \\
    &\leq \norm{V_m r_m(H_m) e_1}_2 
    	 + \norm{r_m(Q) v}_2 \leq 2 (1 + \sqrt{2}) \cdot \epsilon_m . 
\end{align*}

We know, for instance by Weierstrass' theorem, that polynomials approximate
uniformly continuous functions on a compact set, so this alone guarantees
convergence of the scheme. However, it tells us very little
about the convergence speed. It turns out that, for many functions of 
practical use that have a high level of smoothness (such as $f(z) = e^{z}$), 
the convergence is fast. 

When the dimension of the space $\mathcal K_m(Q, v)$ increases, 
the orthogonalization inside the Arnoldi scheme can become
the dominant cost in the method. For this reason, it is advisable to employ
restarting techniques. These aim at stopping the iteration after $m$
becomes sufficiently large, and consider a partial approximation of the
function $f_m^{(1)}$. Then, the residual $f(Q) v - f_m^{(1)}$ is approximated
by restarting the Arnoldi scheme from scratch. The efficient and robust 
implementation of this scheme is non-trivial; we use the approach developed
in \cite{frommer2014efficient}, to which we refer the reader for further
details on the topic. 

\subsection{Approximation of exponential and $\varphi_1(z)$}

After running the Arnoldi scheme, we are left with a simpler problem: we need
to compute $f(H_m) e_1$, where $H_m$ is a small $m \times m$ matrix. In our 
case, we are interested in the functions
\[
  e^{tz} \qquad \text{and} \qquad
  t\varphi_1(tz) = \frac{e^{tz} - 1}{z}. 
\]
The literature on the efficient approximation of the matrix exponential
is vast; the most common approach for $e^{tz}$
is to use a Padé approximation scheme coupled 
with a scaling and squaring technique, which is the default method 
implemented by MATLAB through the function \texttt{expm}; see
the discussion in \cite{higham2008functions} for the optimal
choice of parameters for the scaling phase and the approximation rule. 
Then, one can consider the rational approximant to $e^z$ obtained 
using the Padé scheme with order $(d,d)$, let us call it $r(z)$, 
and so we have 
\[
  e^{H_m} \approx \left( r\left(\frac{1}{2^h} H_m\right) \right)^{2^{h}}. 
\]
The latter matrix power can be efficiently computed by $h$ steps of repeated squaring, 
and the evaluation of the rational function requires $\mathcal O(d)$ 
matrix multiplications and one inversion. The order $d$ has to be chosen
depending on the level of squaring (i.e., on the value of $h$), and 
is an integer between $6$ and $13$ (parameter tuning for
optimal performance and accuracy can be a tricky task, so we suggest
to either refer to \cite{higham2008functions} or to rely on the
MATLAB implementation of \texttt{expm}). 

Concerning the computation of $\varphi_1(z)$, we use a trick widely
used in exponential integrators. 
In particular, 
we propose to recast the problem as the computation of a matrix exponential, 
by exploiting the following known result from the framework of
exponential integrators, whose proof can be found in 
\cite[Theorem 2.1]{al2011computing}. 

\begin{theorem}
	Let $A$ be any $n \times n$ square matrix, and $v \in \mathbb{C}^{n}$. Then, 
	the following relation holds: 
	\[
	  \tilde A := \begin{bmatrix}
	    A & v \\
	    0_{1 \times n} & 0 \\
	  \end{bmatrix}, \qquad 
	  \begin{bmatrix}
	  I_n & 0_{n \times 1}
	  \end{bmatrix} e^{\tilde A}\ \begin{bmatrix}
	  v \\ 0
	  \end{bmatrix} = \varphi_1(A)\ v. 
	\]
\end{theorem}

The above result tells us that the action of $\varphi_1(A)$ on a vector $v$
can be obtained by computing the action of a slightly larger matrix $\tilde A$
on the vector $v$ padded with a final zero. Even when the norm of $A$ is large, 
the techniques in \cite{frommer2014efficient} allow to accurately control the
approximation error. 

\subsection{Incorporating restarting}

In practice, the dimension of the Krylov space needed to achieve
a satisfactory accuracy might be high, and therefore a more 
refined technique is needed to achieve a low computational 
cost. One of the most efficient techniques is to incorporate
a \emph{restarting scheme}: we stop the method after the dimension
of the space reaches a certain maximum allowed dimension, 
obtaining an approximation $f_1(Q) v$ of low-quality. Then, 
we restart the method to approximate $(f - f_1)(Q) v$, i.e., 
the residual. The procedure is then
repeated until convergence. 

An efficient implementation of such scheme
is far from being trivial, and we rely on the restarting
scheme proposed in \cite{frommer2014efficient}, to which
we refer for further details. 
Our implementation relies on the \texttt{funm\_quad} package
that is provided accompanying the paper \cite{frommer2014efficient}. 

\section{Sensitivity analysis}
\label{sec:sensitivity}

Performance and dependability models are often parametric, in the sense that some transition rates $Q_{ij}$ can be functions of some parameter $\lambda$. As described in~\cite{RST89}, the sensitivity analysis is the study of how the measures of interest vary at changing $p$. 

\subsection{A motivation for sensitivity analysis}

During the design phase, a key requirement is to
isolate the set of parameters that most influence the behavior of 
the system. This can guide optimization to the design. In particular, 
this allows to investigate the return of an investment aimed at 
changing some components, in terms of enhanced reliability and/or
availability. When the budget for developing a new product is 
limited, this is of paramount importance. 

Moreover, real world parameters come from actual (physical) measurements and 
therefore might be affected by measure errors of different orders of magnitude,
in particular for cyber-physical systems.
Sensitivity analysis can guide the effort in collecting the most relevant parameters with high precision and the other parameters with acceptable precision.

Another setting where sensitivity analysis plays a relevant role is 
the modeling of complex systems, where certain aspects of the system behavior are often
abstracted away because the time scale at which they appear is considered too fine (avoiding stiffness) or in order to maintain a reasonable level of complexity within the model itself (state space explosion avoidance).
In particular, a hierarchical modeling strategy~\cite{T01} may be adopted: specific system components are modeled in isolation, measures are defined on them and the numerical value obtained evaluating the measures are used as parameters for the overall system model.
The hierarchical strategy can be employed whenever the system logical structure presents a (partial) order among components and can involve several layers.
Establishing to which extent each layer is sensitive to those parameters that come from an underlying layer enhances and guides modeling choices. 

\subsection{Sensibility analysis and Frech\'et derivatives}

We are interested in bounding the first order expansion of a measure
$g(v,w)$ when the infinitesimal generator changes along a certain direction. 
Let $g_{p}(v, w)$ be a performability measure of a system with matrix $Q$ depending
on a parameter $p$ in a smooth way. We want to determine a real 
positive number $M$ such that: 
\[
  |g_{p}(v, w) - g_{p_0}(v,w)| \leq M \cdot |p - p_0| + \mathcal O(|p - p_0|^2). 
\]
This characterizes the amplification of the changes in the system behavior
when the parameter $p$ changes. If $Q$ depends smoothly on $p$ then we can
expand it around $p_0$: 
\[
  Q(p) = Q(p_0) + (p - p_0) \cdot \frac{\partial}{\partial p} Q(p_0) + R(p), \qquad 
  \norm{R(p)} \leq \mathcal O(|p - p_0|^2). 
\]
From now on, by a slight abuse of notation, we will write $O(|p-p_0|^2)$ in place
of $R(p)$, meaning that the bound is correct up to the second order terms in norm.
A straightforward computation yields the following result. 

\begin{lemma}
    Let $Q(p)$ a matrix with a $C^1$ dependency on $p$ around a point $p_0$, and let 
    $g_p(v, w) = v^T f(Q(p)) w$. Then, we have 
    \[
      |g_{p}(v, w) - g_{p_0}(v,w)| \leq \left| v^T D_f(Q(p_0))
        \left[\frac{\partial Q(p_0)}{\partial p}\right] w \right| + 
        \mathcal O(|p - p_0|^2), 
    \]
    where $D_f(\cdot)$ is the Frech\'et derivative of the matrix function $f(\cdot)$. 
\end{lemma}

Even more interestingly, there exists a simple strategy (presented, for example, in 
\cite{higham2008functions}) to compute the Frech\'et derivative
along a certain direction by making use of block matrices. 
Specializing it to our case yields the following corollary. 
\begin{corollary} \label{cor:sensitivity-block}
    Let $Q(t)$ a matrix with a $C^1$ dependency on $p$ around a point $p_0$, and let 
    $g_p(v, w) = v^T f(Q(p)) w$. Then, we have 
    \[
      |g_{p}(v, w) - g_{p_0}(v,w)| \leq \begin{bmatrix}
        v^T & 0^T  \\
      \end{bmatrix} f\left( \begin{bmatrix}
       Q(p_0) & \frac{\partial Q(p_0)}{\partial p} \\
       & Q(p_0)\\
      \end{bmatrix} \right) \begin{bmatrix}
      0 \\
      w \\
      \end{bmatrix} + \mathcal O(|p - p_0|^2), 
    \]
    where the vectors are partitioned accordingly to the $2 \times 2$ block
    matrix. 
\end{corollary}

\begin{proof}
    The result is an immediate consequence of \cite[Theorem 4.12]{higham2008functions}. 
\end{proof}

In view of the above result, if we are given an efficient
method to evaluate $v^T f(Q) w$, the sensitivity with respect to a certain
perturbation of parameters can be computed by extending the method to work
on a matrix of double the dimension. 

For dense linear algebra methods, which have a cubic complexity, this amounts to $8$ times the cost of 
just computing $g(v, w)$; for methods based on quadrature or 
Krylov subspaces, which have a 
linear complexity in the dimension, this means twice the cost of 
an evaluation. 

In both cases, the asymptotic cost for the computation does not increase. 

\section{Numerical tests}
\label{sec:numerical}

In this section we report some practical example of the computation
of availability and performance measures relying on matrix functions. 

The results can be replicated using the MATLAB code that we have
published at \url{https://github.com/numpi/markov-measures}, by running
the scripts \texttt{Example1.m}, \ldots, \texttt{Example4.m}. 
The numbering of the examples coincides with the one of the following subsections. 
The parameters controlling 
the truncation both in the commercial solver employed and 
in \texttt{funm\_quad} are set to $10^{-8}$. For the \texttt{funm\_quad} 
package we have used a restart every $15$ iterations and a maximum
number of restarts equal to $10$, which has never been reached 
in the experiments. 

For instantaneous measures, our tests rely directly on the integral
representation of the matrix exponential in \texttt{funm\_quad}. 
This method is denoted by \texttt{quad\_exp} in tables and figures. 
For cumulative measures, involving the evaluation of $\varphi_1(z)$, 
the methods based on rephrasing the problem as the action of a 
matrix exponential is identified by \texttt{exp\_phi}. 

The tests have been performed
with MATLAB r2017b running on Ubuntu 17.10
on a computer with an Intel i7-4710MQ CPU running
at 2.50 GHz, and with 16 GB of RAM clocked at 1333 MHz. 

\subsection{Average queue length} \label{sec:example1}

We consider a simple Markov chain that models a queue for some service. 
The process $X(t)$ has as possible states the integers 
$\{ 0, \ldots, n - 1\}$, which represent the number of clients in
the queue. A pictorial representation of the states
for $n = 9$ is given in Figure~\ref{fig:markov-1}.

\begin{figure}
	\centering
\begin{tikzpicture}
	\foreach \x in {0, ..., 8} {
		\draw (\x, 0) circle (.3);
		\node at (\x, 0) {$\x$};
	}
	\foreach \x in {1, ..., 8} {
		\draw[->] plot[smooth] coordinates{(\x-.8,.3) (\x-.6,.4) (\x-.4,.4) (\x-.2,.3)};
		\node at (\x-.5,.7) {$\rho_1$};
		\draw[->] plot[smooth] coordinates{(\x-.2,-.3) (\x-.4,-.4) (\x-.6,-.4) (\x-.8,-.3)};
		\node at (\x-.5,-.7) {$\rho_2$};
	}
\end{tikzpicture}
\caption{Pictorial representation of the Markov chain modeling a queue
	for a service, with $n = 9$ states. The rates of the probabilities
	of jumping between the states are reported on the edges.}
\label{fig:markov-1}
\end{figure}
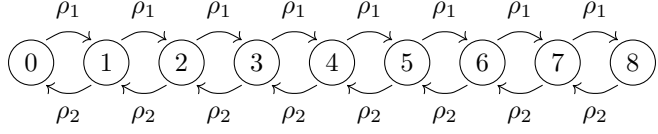

At any state, the rate of probability of jumping 
``left'' (i.e., to serve one client) is equal to $\rho_1$, whereas
the rate of probability at which a new client arrives is equal to
$\rho_2$. The corresponding matrix $Q$ for the Markov chain is 
as follows: 
\[
  Q = \begin{bmatrix}
    -\rho_{2} & \rho_2 \\
    \rho_1    & -(\rho_1 + \rho_2) & \rho_2 \\
    & \ddots & \ddots & \ddots \\
    & & \rho_1 & -(\rho_1 + \rho_2) & \rho_2 \\
    & & & \rho_1 & -\rho_1 \\
  \end{bmatrix} \in \mathbb{C}^{n \times n}. 
\]
According to Section~\ref{sec:clients}, this measure
can be expressed in the form 
\[
  \rinst(t) = \pi_0^T e^{tQ} r, \qquad 
  r = \begin{bmatrix}
    0 \\
    1 \\
    \vdots \\
    n - 1 \\
  \end{bmatrix}. 
\]
The reward vector $r$ gives to each state a weight proportional to
the number of clients waiting in the queue. 
We assume the initial state $\pi_0$ to be the vector $e_1$, corresponding
to starting with an empty queue. The measure gives the average 
expected number of waiting clients at time $t$. 

We have tested our implementation based on the quadrature scheme
described in \cite{afanasjew2008implementation}, and 
the timings needed to compute the
measures as a function of the number of slots in the queue (that is, the
size of the matrix $Q$) are reported in Figure~\ref{fig:clients}. 

The proposed approach has a linear complexity growth as the number $n$ 
increases, as expected.

\begin{figure}
	\centering
	\begin{minipage}{.6\linewidth}
	\begin{tikzpicture}
	    \begin{loglogaxis}[legend pos = north west, 
	      xlabel = \# of slots ($n$), ylabel = Time (s), 
	      width = .9\linewidth, height = 6.5cm]
	     \addplot table {clients.dat};
	     \addplot table[y index = 2] {clients_moeb.dat};
	     \addplot[domain = 1024 : 1000000, dashed] {1e-5 * x};
	     \legend{\texttt{quad\_exp}, M\"obius, $\mathcal O(n)$}
	    \end{loglogaxis}
	\end{tikzpicture}
	\end{minipage}~\begin{minipage}{.35\linewidth}
	  \centering
	  \pgfplotstabletypeset[
	    columns/0/.style = {column name = $n$}, 
	    columns/1/.style = {column name = Time (s)},
	    columns/2/.style = {column name = M\"obius}
	  ]{clients.dat}
	\end{minipage}
	
    \caption{Time needed to approximate the average number of clients at time 
    	$1$ depending on the total number of slots. The sizes range from 
    	$2^{10}$ to $2^{20}$.}
    \label{fig:clients}
\end{figure}
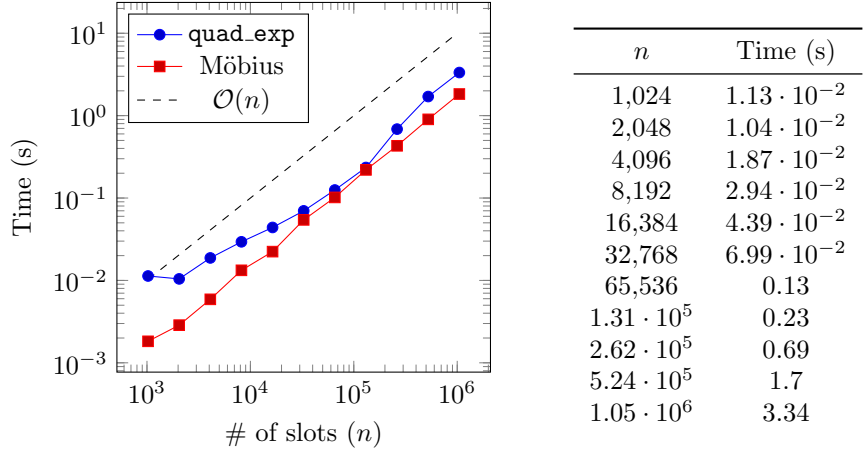

The accuracy requested was set to $10^{-8}$. We note
that, in this example, changing the number of available
slots does not alter this measure in a distinguishable way: the states
with a large index are very unlikely to be reached in a single
unit of time, and therefore have a very low influence
on the distribution $\pi(t)$ with $t = 1$. 

\subsection{Availability modeling for a telecommunication system}
\label{sec:example2}

We consider an example taken from \cite{TB17}[Example 9.15], which describes
a telecommunication switching system with fault detection / reconfiguration
delay. This model describes $n$ components which may fail independently, 
with a mean time to failure of $\frac{1}{\gamma}$. 
After failure of one component, this situation is detected and the entire system switches to \emph{detected mode}. If this happens, 
the component is repaired with a
certain probability $c$ (coverage factor), expected time of $\frac{1}{\delta}$, and the system came back to \emph{normal mode}; 
otherwise, with probability $1 - c$, the component remains failed, the system switches to \emph{normal mode} anyway, and the failed component is repaired with expected time $\frac{1}{\tau}$. 
The pictorial description of the system is reported in Figure~\ref{fig:trivedi}, and the nonzero
structure of the infinitesimal generator $Q$ is reported in Figure~\ref{fig:spyQexample2}.
\begin{figure}
	\centering 
	\includegraphics[width=.9\linewidth]{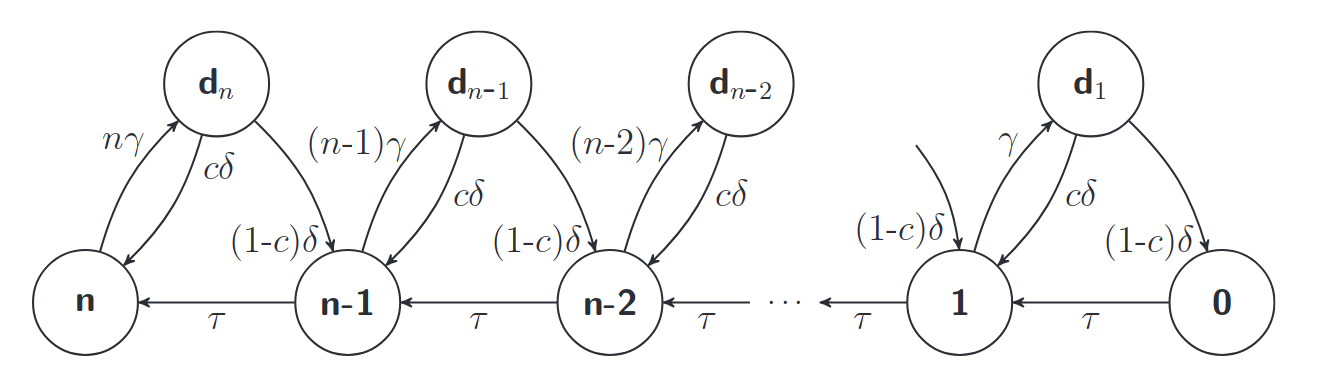}
	
    \caption{Pictorial description of the telecommunication switching 
    	system with fault detection / reconfiguration, modeled through
    	a Markov chain.
	    Figure taken from~\cite{TB17}}
    \label{fig:trivedi}
\end{figure}
\begin{figure}
	\centering
	\includegraphics[width=0.8\linewidth]{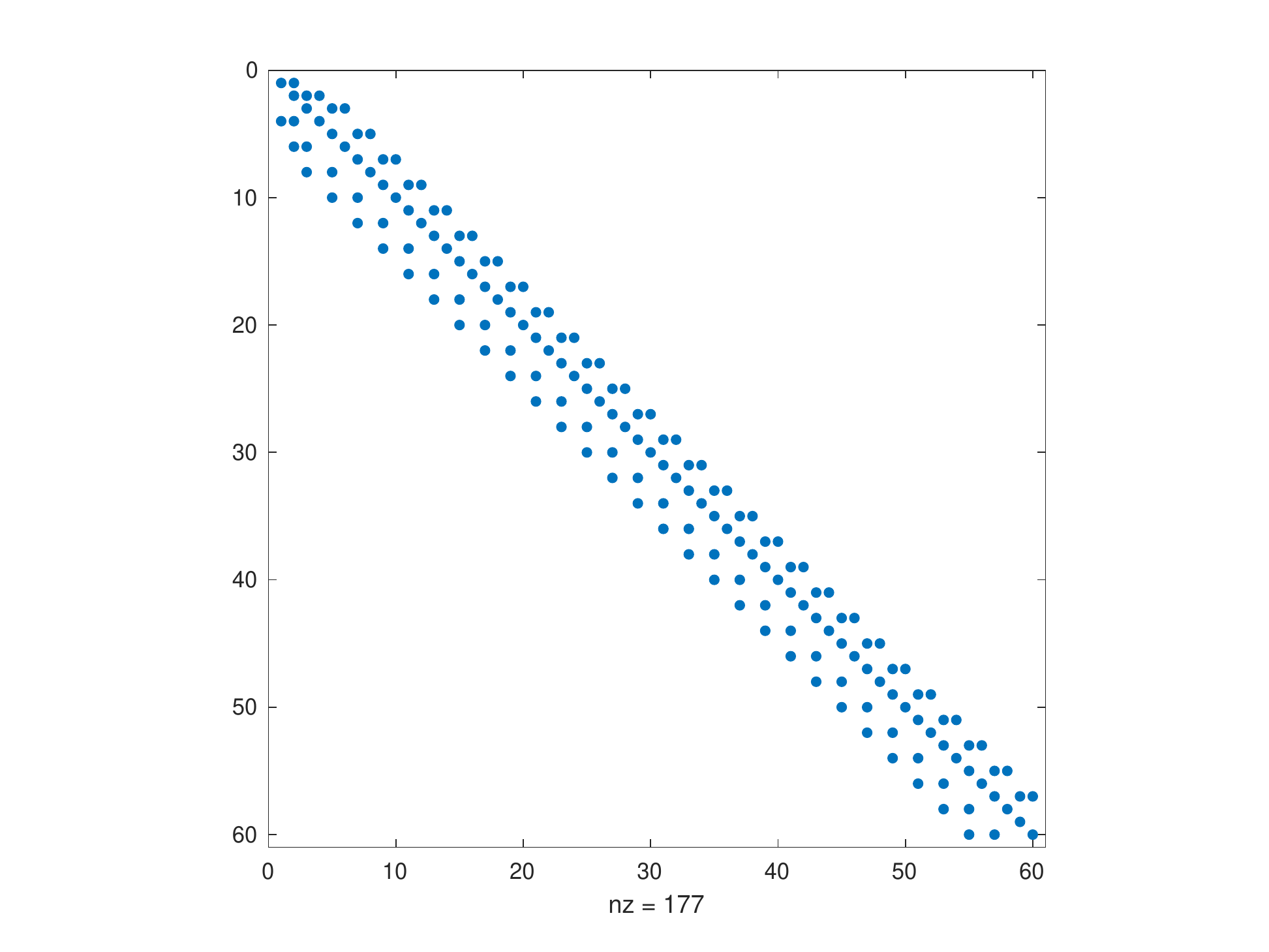}
	\caption{\label{fig:spyQexample2} Non-zero structure of leading $60 \times 60$ minor of the infinitesimal generator $Q$ for the system described
		in Section~\ref{sec:example2}. The matrix $Q$ is banded, and the structure is
	    repeated along the diagonal.}
\end{figure}
This Markov chain is irreducible, and we compute, fixed the interval $[0,t]$, 
the average time the system spends in detected mode from time $0$ to time $t$.

Denoting with $\mathbf d$ the set of detected states, labeled as $d_i$ for $i=1,\dots,n$ in Figure~\ref{fig:trivedi}, this measure, rephrasing what already seen in~Section~\ref{sec:uptime}, can be computed as

\[
D(t) = \int_{0}^t \mathbb E\left[ r_{X(\tau)} \right]\ d\tau , \qquad 
r = \mathbbm{1}_{\mathbf d}. 
\]

In our test, we consider the interval of time $[0, 20]$, i.e., 
$t = 20$. 
The parameters are chosen as follows 
\[
  c = 0.2, \qquad \delta = 0.5, \qquad \gamma = 0.95, \qquad \tau = 1.0.
\]
In order to assess the scalability
of our approach when the number of states grows, we consider
large values of $n$ (even though those may not be common for the
particular situation of a telecommunication system). The number of states
in the Markov chain can be shown to be $2n+1$. The computational
time required to compute the two measures is reported in Figure~\ref{fig:trivedi-times} for different values of $n$ ranging between
$2^{10}$ and $2^{15}$. We compare the timings with the
cumulative solver included in M\"obius 2.5~\cite{DCCDDDSW02} that implement the uniformization method. 

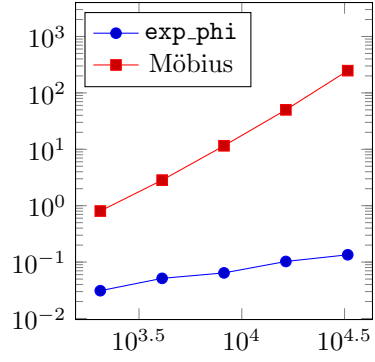
\begin{figure}
	\begin{minipage}{.48\linewidth}
		\centering
		Instantaneous availability \\[5pt]
		\pgfplotstabletypeset[columns={0,2,1},
		columns/0/.style = {column name = $n$},
		columns/2/.style = {column name = M\"obius},
		columns/1/.style = {column name = \texttt{exp\_phi}}
		]{example2.dat}
	\end{minipage}~\begin{minipage}{.48\linewidth}
	\begin{tikzpicture}
	\begin{loglogaxis}[width=.95\linewidth, 
	legend pos = north west, ymax = 4e3, height = .3\textheight]
	\addplot table[x index = 0, y index = 1] {example2.dat};
	\addplot table[x index = 0, y index = 2] {example2.dat};
	\legend{\texttt{exp\_phi}, M\"obius}
	\end{loglogaxis}
	\end{tikzpicture}
\end{minipage}	\caption{Timings for the computation of the cumulative measure in the telecommunication system, described
in Section~\ref{sec:example2}.}
\label{fig:trivedi-times}
\end{figure}

The timings shows that, in this case, the computational complexity
on the solver bundled with M{\" o}bius seems to have a quadratic complexity
in the number of states. This appears to be caused by an increasing
number of iteration needed to reach convergence due to relevant differences among rates in the Markov chain, which
causes stiffness in the underlying ODE. The Krylov approach, 
on the other hand, does not suffer this drawback. 

\subsection{Reliability model for communication system attacks}
\label{sec:example3}

\begin{figure}
	\label{fig:example3}
	\centering{
		\begin{tikzpicture}[node distance=6em]
		\tikzstyle{place}=[circle,thick,draw=black!75,
		minimum size=1.5em]
		\tikzstyle{red place}=[place,draw=red!75,fill=red!20]
		\tikzstyle{transition}=[rectangle,thick,draw=black!75,
		minimum height=2.5em]
		\tikzstyle{flatTransition}=[rectangle,thick,draw=black!75,
		minimum height=.5em,minimum width=2.5em]
		
		\node [place,tokens=0, label=above:$N_G$] (NG) {};
		\node [draw=none] (initialMarking) at (NG) {$N$};
		\node [transition, label=above:$T_{GB}$, label=below:\textcolor{red}{$N_G\lambda_c$}] 
		(TGB) [right of=NG] {};
		\node [place,tokens=0, label=above:$N_B$] (NB) [right of=TGB] {};
		\node [transition, label=above:$T_{BF}$, label=below:\textcolor{red}{$N_Bp_a\lambda_f$}] 
		(TBF) [right of=NB] {};
		\node [place,tokens=0, label=above:$N_F$] (NF) [right of=TBF] {};
		\node[draw=none] (action1) at ($.5*(TBF)+.5*(NF)+(0,1em)$) {\textcolor{blue}{$N_G=0$}};
		\node[draw=none] (action2) at ($.5*(TBF)+.5*(NF)+(0,-1em)$) {\textcolor{blue}{$N_B=0$}};
		\node [flatTransition, label=left:$T_{BE}$, label=right:\textcolor{red}{$N_B\frac{1-P_{\text{fn}}}{T_{\text{IDS}}}$}] 
		(TBE) [below of=NB] {};
		\node [place,tokens=0, label=right:$N_E$] (NE) [below of=TBE] {};
		\node [transition, label=above:$T_{GE}$, label=below:\textcolor{red}{$N_G\frac{P_{\text{fp}}}{T_{\text{IDS}}}$}] 
		(TGE) [left of=NE] {};
		
		\draw [->] (NG) to (TGB);
		\draw [->] (TGB) to (NB);
		\draw [->] (NB) to (TBF);
		\draw [->] (TBF) to (NF);
		\draw [->] (NB) to (TBE);
		\draw [->] (TBE) to (NE);
		\draw [->] (TGE) to (NE);
		\draw [->] (NG) -- (NG|-TGE) -- (TGE);
		\end{tikzpicture}
	}
	\caption{Attack model for the cyber-physical communication system described in Section~\ref{sec:example3}. This model is a simplified version of the model presented in~\cite{MTC17}. Places are represented as circles, transitions are represented as rectangles. Place and transition names are in black, transition rates are in red and actions performed whenever transition $T_{BF}$ completes are in blue.}
\end{figure}
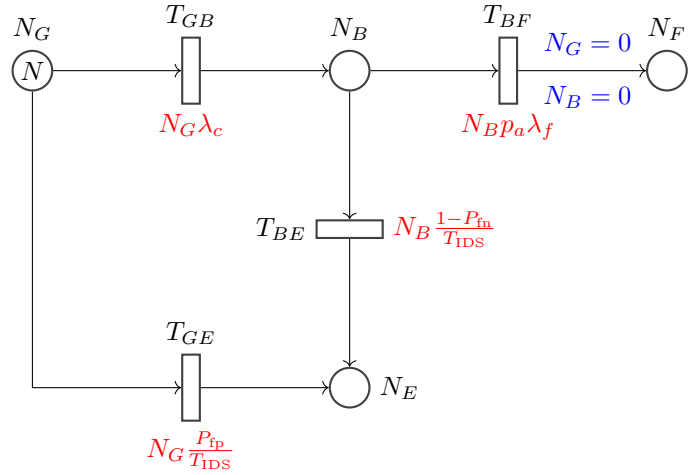

We consider the mobile cyber-physical system model presented in~\cite{MTC17}, describing a collection of communicating nodes which are subject to attacks.
The original study is based on a real-world architecture: there are $N$ mobile nodes, each node using sensors for localization and measuring anomaly phenomena, and the system comprises an imperfect intrusion/detection functionality distributed
to all nodes for dealing with both intrusion and fault tolerance.
This mechanism is based on a voting system.
Here a simplified version is discussed, we refer
the reader to \cite{MTC17} for further details on the intrusion/detection functionality and the complete system description.

The model considers a node capture
which involves taking control of a \emph{good} node by 
deceiving the authentication and turning it
into a \emph{bad} node that
will be able to generate attacks within the system. 
The attackers primary objective is to cause impairment failure by performing
persistent, random, or insidious attacks.
At each instant of time, the number of good and bad nodes are indicated as $N_G$ and $N_B$, respectively, and $N_E$ is the number of \emph{evicted} nodes, i.e., nodes that have been detected as bad ones by the intrusion/detection mechanism.
At the beginning, all nodes are considered good, i.e., $N_G=N$.
Only bad nodes can perform internal attacks, and whenever one of this attacks have success
the entire system fails, switching the value of $N_F$ from $0$, \emph{ok}, to $1$, \emph{failed}.

The model is expressed through the definition of the Stochastic Reward Net depicted in Figure~\ref{fig:example3}, where places (circles) correspond to $N_G,N_B,N_F,N_E$,
and determine the state of the system, and transitions (rectangles) define the behaviour
of the attach model
\begin{itemize}
	\item transition of a node from good to bad, called $T_{GB}$, represents the capture
	of a node by an attacker. 
	The capture of a single node take place with rate $\lambda_c$, thus, being the capture
	of a node independent from the capture of other nodes, the rate of transition $T_{GB}$
	is $N_G\lambda_c$. 
	\item transition of a node from bad to evicted, called $T_{BE}$, represents 
	the correct detection of an attack. 
	Calling $P_{\text{fn}}$ the probability of intrusion/detection false negative,
	and $T_{\text{IDS}}$ the period at which the intrusion/detection mechanism is exercised,
	the rate of $T_{BE}$ is $N_B\frac{1-P_{\text{fn}}}{T_{\text{IDS}}}$.
	\item transition of a node from good to evicted, called $T_{GE}$, represent a false positive
	of the intrusion/detection mechanism.
	Calling $P_{\text{fp}}$ the probability of intrusion/detection false positive,
	the rate of $T_{GE}$ is $N_G\frac{P_{\text{fp}}}{T_{\text{IDS}}}$.
	\item transition of a the entire system from ok to failed, called $T_{BF}$. 
	When a node is captured it will perform attacks with a probability $p_a$ and
	the success of attacks from $N_B$ compromised nodes has rate $\lambda_f$, thus
	the rate of $T_{BF}$ is $N_Bp_a\lambda_f$.
	At completion of transition $T_{BF}$ the entire system fails and then both $N_G$ and $N_B$
	are set to $0$ so that the Stochastic Reward Net reach a (failed) absorbing state.
\end{itemize} 
The graph whose vertexes are all the feasible combinations of values within places and arcs correspond to transitions forms the Markov chain under analysis.
For instance, with $N=3$ The Stochastic Reward Net of Figure~\ref{fig:example3} produces the Markov chain depicted in Figure~\ref{fig:SS}, where the notation $(n_G,n_B,n_E,n_F)$ means
$N_G=n_G, N_B=n_B, N_E=n_E$ and $N_F=n_F$. The nonzero structure of the infinitesimal
generator $Q$ is reported in Figure~\ref{fig:spyQexample3}.
\begin{figure}
	\label{fig:SS}
	\centering {
		\tiny
		\begin{tikzpicture}[scale=1.0]
		\tikzset{absorbingState/.style={circle,draw,fill=black!10}}
		\tikzset{votingState/.style={circle,draw,fill=blue!10}}
			\node[draw=none] (c) at (0,0) {};
			\node[votingState] (1) at ($(c) + (6em,-6em)$) {$3,0,0,0$}; 
			\node[votingState] (2) at ($(c)+(-3em,-6em)$) {$2,1,0,0$}; 
			\node[votingState] (3) at ($(c)+(3em,-12em)$) {$2,0,1,0$}; 
			\node[state] (4) at ($(c)+(-6em,-12em)$) {$1,2,0,0$}; 
			\node[absorbingState] (5) at ($(c)+(-11em,-24em)$) {$0,0,0,1$}; 
			\node[state] (6) at ($(c)+(3em,-18em)$) {$1,1,1,0$}; 
			\node[state] (7) at ($(c)+(8em,-24em)$) {$1,0,2,0$}; 
			\node[state] (8) at ($(c)+(-6em,-18em)$) {$0,3,0,0$}; 
			\node[state] (9) at ($(c)+(-3em,-24em)$) {$0,2,1,0$}; 
			\node[absorbingState] (10) at ($(c)+(-4em,-30em)$) {$0,0,1,1$}; 
			\node[state] (11) at ($(c)+(4em,-30em)$) {$0,1,2,0$}; 
			\node[absorbingState] (12) at ($(c)+(14em,-20em)$) {$0,0,3,0$}; 
			\node[absorbingState] (13) at ($(c)+(14em,-30em)$) {$0,0,2,1$}; 
			\draw[every loop]
				(1) edge[bend right] node {$T_{GB}$} (2)
				(1) edge[bend left] node {$T_{GE}$} (3)
				(2) edge[bend right] node {$T_{BF}$} (5)
				(2) edge[bend left] node {$T_{GB}$} (4)
				(2) edge[bend left] node {$T_{GE}$} (9) 
				(2) edge[bend left] node {$T_{BE}$} (3)
				(3) edge node {$T_{GB}$} (6)
				(3) edge[bend left] node {$T_{GE}$} (7)
				(4) edge[bend right] node {$T_{BF}$} (5)
				(4) edge node {$T_{GB}$} (8)
				(4) edge[bend right] node {$T_{BE}$} (6)
				(6) edge[bend left] node {$T_{GB}$} (9)
				(6) edge[bend left] node {$T_{BF}$} (10) 
				(6) edge node[bend right] {$T_{GE}$} (11) 
				(6) edge[bend left] node {$T_{BE}$} (7) 
				(7) edge[bend left] node {$T_{GB}$} (11) 
				(7) edge[bend left] node {$T_{GE}$} (12) 
				(8) edge[bend right] node {$T_{BF}$} (5) 
				(8) edge[bend left] node {$T_{BE}$} (9) 
				(9) edge[bend right] node {$T_{BF}$} (10) 
				(9) edge node {$T_{BE}$} (11) 
				(11) edge[bend right] node {$T_{BE}$} (12) 
				(11) edge node {$T_{BF}$} (13);
		\end{tikzpicture}
	}
	\caption{Markov chain produced by the Stochastic Reward Net depicted in Figure~\ref{fig:example3}. The initial state is $(3,0,0,0)$, 
	the absorbing states are colored in gray, states such that $N_G\ge 2 N_B$ are colored in blue.} 
\end{figure}
\begin{figure}
	\centering
	\includegraphics[width=0.8\linewidth]{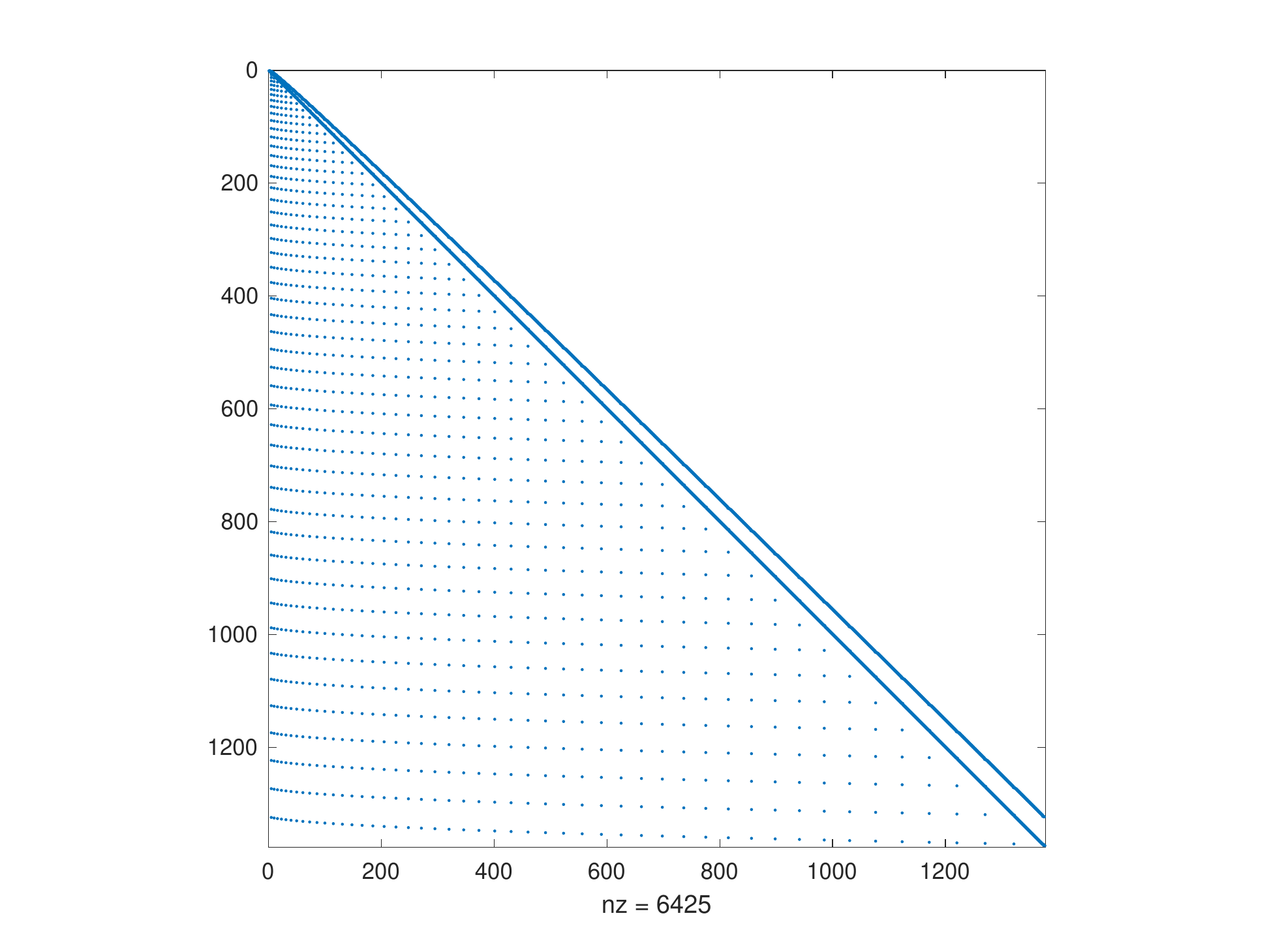}
	\caption{\label{fig:spyQexample3} Nonzero structure of the infinitesimal generator $Q$ for the system described
		in Section~\ref{sec:example3}, where $n=1376$.}
\end{figure}
The parameters are chosen as follows 
\[
p_a = 0.7,\qquad P_{\text{fn}} = P_{\text{fp}} = 0.1, \qquad T_{\text{IDS}} = 15.0, \qquad \lambda_c = 0.1, \qquad \lambda_f = 0.2.
\]

Being the intrusion/detection mechanism based on a voting system, the Byzantine fault model is
selected to define the \emph{security failure} of the system, i.e., the situation in which the system is working but there are not enough good nodes to obtain consensus when voting, that in our system means $N_G < 2 N_B$.
Thus, the cumulative measure of interest is
\[
B_{\text{security}}=\int_{0}^t \mathbb E\left[ r_{X(\tau)}\right]\ d\tau , \qquad 
r = \mathbbm{1}_{\{N_G\ge 2 N_B,N_F=0\}}.
\]

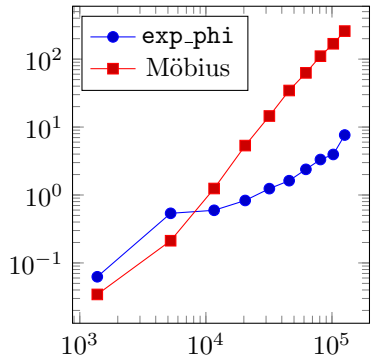
\begin{figure}
	\begin{minipage}{.48\linewidth}
		\centering
		Instantaneous availability \\[5pt]
		\pgfplotstabletypeset[columns={0,2,1},
		columns/0/.style = {column name = $n$},
		columns/2/.style = {column name = M\"obius},
		columns/1/.style = {column name = \texttt{exp\_phi}}
		]{example3.dat}
	\end{minipage}~\begin{minipage}{.48\linewidth}
		\begin{tikzpicture}
		\begin{loglogaxis}[width=.95\linewidth, 
		legend pos = north west, ymax = 6e2, height = .3\textheight]
		\addplot table[x index = 0, y index = 1] {example3.dat};
		\addplot table[x index = 0, y index = 2] {example3.dat};
		\legend{\texttt{exp\_phi}, M\"obius}
		\end{loglogaxis}
		\end{tikzpicture}
	\end{minipage}	\caption{Timings for the computation of the cumulative 
		measure in the security system, described
		in Section~\ref{sec:example3}.}
	\label{fig:security-times}
\end{figure}

We have tested our implementation and measured the time 
required to compute this measure. The results are reported 
in Figure~\ref{fig:security-times}. Beside some overhead when dealing
with small dimensions (and times below $0.01$ seconds), the
approach relying on the restarted Krylov method (labeled by \texttt{exp\_phi}) is faster than the uniformization method included in M\"obius. 

\subsection{Sensitivity analysis}

As a last example, we consider the case of a sensitivity analysis. For simplicity, 
we consider once more the model of Section~\ref{sec:example1}, and we assume
to be interested in changing the parameter $\rho_2$. The only ingredient missing
is computing the derivative of $Q$ with respect to $\rho_2$, which in this case
is simply given by the matrix
\[
  \frac{\partial Q}{\partial \rho_2} = \begin{bmatrix}
  -1 & 1 \\
   & \ddots & \ddots \\
   & & -1 & 1 \\
   & & & 0 \\
  \end{bmatrix}. 
\]
We have implemented the function computing the derivative of the measure
following the approach described in Section~\ref{sec:sensitivity}, 
and the performance of the algorithm is reported in Figure~\ref{fig:sensitivity}, 
for the time $T = 1$. We see that also in this case the scalability
of the asymptotic cost of the 
algorithm is linear with respect to the number of states, as
expected. 

\begin{figure}
	\centering
	\begin{minipage}{.6\linewidth}
		\bigskip 
		\begin{tikzpicture}
		\begin{loglogaxis}[legend pos = north west, 
		xlabel = \# of slots ($n$), ylabel = Time (s), 
		width = .9\linewidth, height = 5cm]
		\addplot table {sensitivities.dat};
		\addplot[domain = 512 : 1024000, dashed] {1e-5 * x};
		\legend{\texttt{quad\_exp}, $\mathcal O(n)$}
		\end{loglogaxis}
		\end{tikzpicture}
	\end{minipage}~\begin{minipage}{.35\linewidth}
	\centering
	\pgfplotstabletypeset[
	columns/0/.style = {column name = $n$}, 
	columns/1/.style = {column name = Time (s)}
	]{sensitivities.dat}
\end{minipage}

\caption{Time needed to approximate the sensitivity of the
	average number of clients at time 
	$1$ depending on the total number of slots. The sizes range from 
	$2^{8}$ to $2^{20}$.}
\label{fig:sensitivity}
\end{figure}

As demonstrated by the experiments, the Krylov based approach often outperforms
the classical uniformization method. This could be explained by the fact that
the Krylov method performs the 
number of restarts required to achieve a certain accuracy given a specific choice
of $\pi_0^T$ and $r$; in the case of the uniformization, instead, one has to choose a small
timestep when dealing with stiff problems, ignoring the effect of the left
and right vectors $\pi_0$ and $r$. 

\section{Conclusions}

We have presented a novel point of view on the formulation
of availability, reliability, and performability measures
in the setting of Markov chains. The main contribution
of this work is to provide a systematic way to rephrase 
these measures in terms 
of bilinear forms defined by appropriate matrix functions. 

A dictionary translating the most common measures in the
field of Markov modeling to the one of matrix functions
has been described. We have proved that, leveraging the
software available for the efficient computation of the
action of $f(Q)$ on a vector, we can easily devise a machinery
that evaluates these measures with the same or better
performances that states of the art solvers, such as 
the one included in M\"obius (which implements
the uniformization method). 

In particular, our solver seems to be more robust to
unbalanced rates in the matrix and stiffness
in general, which can make a dramatic
difference in some cases, as showcased by our numerical experiments. 

The new formulation allows to study measures'
 sensitivity by a new perspective, namely Frechét derivatives. 
 This appears to be a promising reformulation that, along 
 with providing efficient numerical procedures, might
 give interesting theoretical insights in the future. 
 
 We expect that this new setting will allow to devise
 efficient method tailored to the specific structure of 
 Markov chains arising, for instance, from different
 high-level modeling languages. 
 
 Several problems remain open for further study. For instance, 
 we have analyzed the use of Krylov methods with restarts, 
 but the use of rational Krylov methods appear promising
 as well. Moreover, often the infinitesimal generator have
 particular structures induced by the high-level formalism
 used to model the Markov chain --- as it has often
 been noticed in the literature --- which we have not exploited here. 
 These topics will be subject to future investigations. 

\appendix 

\renewcommand{\thesection}{Appendix~\Alph{section}}
\section{Continuous time Markov chains and reward structures}
\renewcommand{\thesection}{\Alph{section}}

Markov models are typically defined using high level formalism and then translated into \acp{CTMC}. Most often, the
modeler is interested in extracting relevant information on the Markov
chain, such as probability of breakdown (or, from the opposite perspective, 
of completing operations \emph{without} breakdown). 
This kind of information is described abstractly using a \emph{reward structure}, 
which is defined at the higher level --- using the \ac{MRP} language. These
measures assess the performance and dependability of the system, and 
so are called \emph{performability measures}~\cite{M82}. 

\subsection{Definition of models and measures}
\label{subsec:model}

Given a CTMC, a natural question is if the limit for $t \to \infty$
of the probability distribution $\pi(t)$ exists, and whether it depends on
the initial choice $\pi_0$. To characterize this behavior, we need to
introduce the concepts of \emph{irreducibility}. 

\begin{definition}
	Let $X(t)$ be a Markov chain with infinitesimal generator $Q$, and consider
	the directed graph $\mathcal G$ with nodes the set of states of $X(t)$, and 
	with an edge from $i$ to $j$ if and only if $Q_{ij} > 0$. We say that 
	$X(t)$ is \emph{irreducible} if for for every two states $i, j$ there
	exists a path connecting $i$ to $j$. We say that $X(t)$ is
	\emph{reducible} if it is not \emph{irreducible}. 
\end{definition}

We shall partition the Markov processes in two classes: the 
irreducible ones, called \emph{transient}, which in the finite case are also \emph{positive recurrent}, 
and the reducible ones, which are called \emph{terminating}.

Intuitively, a Markov chain is irreducible if there is always a nonzero
probability of jumping from $i$ to $j$, possibly through some intermediate
jumps. This property is sufficient to guarantee the existence and uniqueness
of the \emph{steady-state} vector $\pi$, the limit of $\pi(t)$ for 
$t \to \infty$. 

From the linear algebra point of view, it is often useful to notice that 
a Markov process is reducible if and only if the matrix $Q$ can be made
block upper triangular by permuting the rows and column (with the same
permutation).

We refer to the book \cite{TB17} for a more detailed analysis on the classification
of Markov chains, which we do not discuss further. 

\begin{theorem}
	Let $X(t)$ be an irreducible Markov chain with a finite number of states
	and infinitesimal generator $Q$. Then, there exists a unique positive
	vector $\pi$ such that $\pi^T = \lim_{t \to \infty} \pi_0^T e^{tQ}$, 
	and $\pi$ does not depend on $\pi_0$. Moreover, $\pi^T$ generates
	the left kernel of $Q$. 
\end{theorem}

\subsection{Spectral properties of $Q$ and the steady-state}
\label{sec:spectral}

As we already pointed out, unless the Markov chain is irreducible, 
the steady state probability might not be unique --- and therefore
depend on the specific choice of initial configuration $\pi_0$. 
The uniqueness can be characterized by considering the spectral 
properties of $Q$.

\begin{lemma}
    Let $Q$ be the infinitesimal generator of a continuous time
    Markov chain. Then, $Q$ is singular, and has $\mathbbm 1$ as right
    eigenvector corresponding to the eigenvalue $0$; if
    the Markov process is irreducible, then the left kernel is
    generated by $\pi^T$, the steady-state distribution. 
\end{lemma}

The matrix $Q$ has another distinctive feature from the linear
algebra point of view, which we will use repeatedly in what follows. 

\begin{lemma} \label{lem:m-matrix}
	Let $Q$ be the infinitesimal generator of a Markov chain. 
	Then,  $-Q$ is an $M$-matrix, i.e., there exists a positive $\alpha$ 
    such that 
    \[
      M = \hat M - \alpha I, 
    \]
    with $\hat M$ being a non-negative matrix with spectral radius
    bounded by $\alpha$. 
\end{lemma}

Lemma~\ref{lem:m-matrix} implies, by a straightforward application 
of classical Gerschgorin theorems \cite{varga2010gervsgorin}, that 
all the eigenvalues of $Q$ are contained in the left half 
of the complex plane. 

The spectral features of $Q$ are tightly connected 
with the asymptotic behavior
of the Markov chain. In particular, as we have
pointed out in the previous section, a reducible 
process has a matrix $Q$ that can be permuted to be
block upper triangular. More precisely, one can reorder
the entries as 
\begin{equation} \label{eq:permutedprocess}
  \Pi Q \Pi^T = \begin{bmatrix}
    Q_{\mathbf u} & Q_{\mathbf u \mathbf d} \\
    Q_{\mathbf d \mathbf u} & Q_{\mathbf d} \\
  \end{bmatrix}, \qquad 
  Q_{\mathbf d \mathbf u} = 0, 
\end{equation}
where $\Pi$ is a permutation that lists the indices in $\mathbf u$ 
first, and then the ones in $\mathbf d$. The set $u$ are the
transient states of the process, whereas $\mathbf d$ contains the
recurrent ones. The matrix $Q_{\mathbf u \mathbf d}$ contains the probability
rates of jumping from a state in $\mathbf u$ to a state in $\mathbf d$. 
The following characterization will be relevant in the following. 

\begin{lemma} \label{lem:block-invertible}
    Let $Q$ be the infinitesimal generator of a terminating process, 
    and $\Pi$ the permutation identified in \eqref{eq:permutedprocess}. 
    Then, $Q_{\mathbf u}$ is invertible. 
\end{lemma}

\begin{proof}
	We claim that every row of $Q_{\mathbf u}$ has the sum of the
	off-diagonal elements strictly smaller than the modulus
	of the diagonal one. We assume that the matrix has been permuted
	already, so we may write without loss of generality:
	\[
	  \mathbf u = \{ 1, \ldots, i' \}, \qquad 
	  \mathbf d = \{ i' + 1, \ldots, n \}. 
	\]
	We have that, for every $i \leq i'$: 
	\[
	  \sum_{j \leq i'} ( Q_{\mathbf u} )_{ij} + \sum_{i'+1 \leq j \leq n}
	    (Q_{\mathbf u \mathbf d})_{i, j-i'} = 0 \implies
      \sum_{j \leq i', j \neq i} ( Q_{\mathbf u} )_{ij} + (Q_{\mathbf u})_{ii} < 0. 
	\]
	Considering that the diagonal element has negative sign, and 
	all the others are positive, we conclude that 
	\[
	  |(Q_{\mathbf u})_{ii}| > \sum_{j \leq i', j \neq i} |( Q_{\mathbf u} )_{ij}|, 
	\]
	and therefore $0$ is not included in any of the Gerschgorin circles of 
	$Q_{\mathbf u}$, and $Q_{\mathbf u}$ is invertible \cite{varga2010gervsgorin}. 
\end{proof}

\subsection{Available solution methods}
\label{sec:STDmethods}

Solving the Markov chain means computing the probability vector $\pi(t)$
at a given time $t$ or, if the chain is irreducible, computing the steady-state probability vector $\pi$. 
As discussed in the following, often it is required to compute also $\int_0^t \pi_i(\tau) d\tau$ for some index $i$. We
recall in this section the most well-known methods to tackle
this task in the generic case (without particular
assumption on the Markov chain). 

A simple approach is the direct computation of the matrix
exponential $e^{tQ}$, which in turns allows 
to obtain $\pi(t)^T = \pi_0^T e^{tQ}$. However, this is 
only feasible if the number of states is small, because
the complexity is cubic in the number of states. 

Another strategy is to compute the Laplace transform~\cite{TB17} of $\pi(t)$, denoted by $\bar{\pi}(s)$, solving the linear system $\bar{\pi}^T(s)\big(sI-Q\big)=\bar{\pi}^T_0$, obtained applying the Laplace transform to the Kolmogorov forward equation, and then anti-transform $\bar{\pi}(s)$ producing $\pi(t)$. 
The parameter $s$ can be chosen so that $sI-Q$ is non-singular, the linear system can be solved exploiting favorable properties of $Q$, e.g., sparseness, and the known term $\bar{\pi}_0$ is often easy to compute because $\pi_0$ in dependability and performance models is highly structured.
This method can be applied only to relatively small chains, being the anti-transform a costly operation, but can tackle relatively large $t$.
Numerical integration~\cite{RT89} from zero to $t$ of the Kolmogorov forward equation and/or of 
\[
\begin{cases}
\dot{L}(t)&= L^T(t) Q+\pi^T_0\text{,}\\
L(0) &= 0\text{,}
\end{cases}
\]
where $L=\int_0^t \pi(\tau)d\tau$, is another alternative.
Unfortunately, in dependability models the parameters can have different order of magnitudes, e.g., in a cyber-physical system the mean time to failure of an hardware component is considerably different from the mean time to failure of a software component, and similarly in performability models the performance-oriented and fault-related parameters can have different scalings.
Thus, numerical integration is a good choice only if the chosen method, for the particular case under analysis, has been proved to be highly resilient to stiffness.
For general chains and arbitrary $t$, the \emph{uniformization method}~\cite{RT89}
is commonly adopted by commercial level software tool, such as
M\"obius \cite{DCCDDDSW02}. 
An heuristically chosen $q$ such that $q>\max_i |Q_{ii}|$ allows to compute the truncate series expansion of $\pi(t)$ and $\int_0^t\pi(\tau)d\tau$ in terms of powers of $Q^*=(\frac{Q}{q}+I)$. In the numerical experiments, 
we will compare the performances of the
approach proposed in this paper with the solver bundled with M\"obius. 

\subsection{Matrix functions} \label{sec:basics}

Matrix functions are ubiquitous in applied mathematics, and appear
in diverse applications. We refer the reader to \cite{higham2008functions}
and the references therein for more detailed information. The definition
of a matrix function can be given in different (equivalent) ways. Here
we recall the one based on the Jordan form. 

\begin{definition}[Matrix function] \label{def:matfun}
    Let $A$ be a matrix with spectrum $\sigma(A) = \{ \lambda_1, \ldots, \lambda_n \}$, and $f(z)$ a function that is analytic on the spectrum
    of $A$. Let $J = V^{-1} A V$ be the Jordan form of $A$, 
    with $J = J_1(\lambda_{j_1}) \oplus \ldots \oplus J_k(\lambda_{j_k})$ being its decomposition
    in elementary Jordan blocks; we define
    the \emph{matrix function} $f(A)$ as 
    \[
      f(A) = V f(J) V^{-1}, \qquad 
      f(J) = f(J_1) \oplus \ldots \oplus  f(J_k)
    \]
    where for an $m \times m$ Jordan block we have:
    \[
      f(J(\lambda)) = \begin{bmatrix}
       f(\lambda) & f'(\lambda) & \ldots & f^{(m)}(\lambda) \\
       & \ddots & \ddots & \vdots \\
       & & \ddots & f'(\lambda) \\
       & & & f(\lambda)
      \end{bmatrix}
    \]
\end{definition}

Definition~\ref{def:matfun} is rarely useful (directly)
from the computational point 
of view. In most cases, computation of matrix functions 
is performed relying on the Schur form, on block diagonalization
procedures, on contour integration, 
or on rational and polynomial approximation (we refer to
\cite{higham2008functions} and the references therein for
a comprehensive analysis of advantages and disadvantages of
the different approaches). 

\begin{remark} \label{rem:weak-matfundef}
	Note that, in order to apply Definition~\ref{def:matfun}, we do not
	necessarily need $f(z)$ to be analytic on the whole spectrum of $A$; 
	we just need $f$ to have derivatives of order $m$ at every point corresponding
	to an eigenvalue with Jordan blocks of size at most $m + 1$. 
\end{remark}

Some examples of matrix functions that appear frequently
in applied mathematics are
the matrix exponential $e^{A}$, the inverse $A^{-1}$
and the resolvents $(sI - A)^{-1}$, 
the square root $A^{\frac 12}$ and the matrix logarithm $\log(A)$. 

\renewcommand{\thesection}{Appendix~\Alph{section}}
\section{Rephrasing the measures} \label{sec:rephrasing}
\renewcommand{\thesection}{\Alph{section}}

This section aims to provide the building blocks that enables
the translation of the measures from the Markov chain setting, 
where they are expressed as expected values of a random
variable obtained as a function of $X(t)$, to the more 
computationally-friendly matrix function form.

\begin{proof}[Proof of Lemma~\ref{lem:acc}]
    Recall that, by Definition~\ref{def:inst-perfomability}, we have 
    \[
      \rcum(t) = \int_{0}^t \rinst(\tau)\ d\tau = 
      \pi_0^T \left( \int_{0}^t e^{\tau Q}\ d\tau \right) r, 
    \]
    by linearity of the integral. Assume, for
    simplicity, that $Q$ is diagonalizable, and let 
    $Q = VDV^{-1}$ be its eigendecomposition; 
    we can write 
    \[
      \rcum(t) = \pi_0^T V \left( \int_0^{t} e^{\tau D}\ d\tau \right)V^{-1} r = 
      \pi_0^T V f(D) V^{-1} r, \quad 
      f(z) = \int_{0}^t e^{\tau z}\ d\tau. 
    \]
	By direct integration we get $f(z) = \frac{e^{\tau z} - 1}{z}$. 
	Finally, using the relation $V f(A) V^{-1} = f(VAV^{-1})$ we obtain 
	the sought equality 
	$\rcum(t) = \pi_0^T f(Q) r$. The general statement follows by density of 
	diagonalizable matrices, together with the continuity of $f(z)$. 
\end{proof}

\begin{remark}
	We shall note that $t\varphi_1(tz)$ 
	defined in Lemma~\ref{lem:acc}, 
	despite being defined piece-wise, is an analytic function, and 
	is in fact defined by the power-series
	\[
	  t\varphi_1(tz) = t \cdot \left( 1 + \frac{tz}{2} + \frac{(tz)^2}{3!} + \ldots 
	   + \frac{(tz)^j}{(j+1)!} + \ldots \right), 
	\]
	which is convergent for all $z \in \mathbb C$. 
	Nevertheless, the formula $(e^{tz} - 1)/z$ cannot be used
	directly to evaluate the function at $Q$, because is not well-defined
	when $z = 0$, and this point is always included in the spectrum, 
	since $Q$ is singular. The notation $\varphi_1(z) = (e^z - 1)/z$ 
	is often used
	in the description of exponential integrators. We refer to \cite{gander2013paraexp}
	and the references therein for more details. 
\end{remark}

A similar statement can be given also for the steady-state case. Nevertheless, 
to achieve this we need to consider a discontinuous function, and  
this can cause difficulties in the numerical use of such 
characterization. 

\begin{lemma} \label{lem:steady-state}
    Let $X(t)$ a continuous irreducible Markov process with infinitesimal generator
    $Q$, and assume an initial
    distribution of probability $\pi_0$. The steady-state distribution
    $\pi$ can be expressed as follows:
    \[
      \pi^T = \pi_0^T \delta(Q), \qquad 
      \delta(z) = \begin{cases}
        1 & \text{if } z = 0 \\
        0 & \text{otherwise}
      \end{cases}
    \]
\end{lemma}

\begin{proof}
    We notice that, for any finite time $t$, $\pi(t) = e^{tQ}$. Since the
    spectrum of $Q$ is contained in $\{ \Re(z) < 0 \} \cup \{ 0 \}$, 
    it is sufficient to check that 
    $f_t(z) = e^{tz}$ converges to $\delta(z)$ as $t \to \infty$ 
    on this set. It is clear that, for every $z \neq 0$ 
    in the left half plane, we indeed have $f_t(z) \to 0$, and 
    the claim
    follows noting that $f_t(0) = 1 = \delta(z)$ independently of $t$. 
\end{proof}

\begin{remark}
	The fact that $\delta(z)$ is not analytic on the
	spectrum of $Q$ is not an obstruction in the application 
	of Definition~\ref{def:matfun}. In fact, since $0$ is
	a simple eigenvalue, the definition is still applicable 
	in view of Remark~\ref{rem:weak-matfundef}. 
\end{remark}

\begin{remark}
    If the process is irreducible
    then the matrix function $\delta(Q)$ takes the form 
    $\delta(Q) = \mathbbm 1 \pi^T$ and therefore, since $\pi_0^T \mathbbm 1 = 1$ for every probability distribution $\pi_0$, it is clear that the
    steady-state is independent of $\pi_0$. 
\end{remark}

 
\bibliographystyle{elsarticle-harv}
\bibliography{mfun-paper}

\end{document}